\documentclass[11pt, reqno]{amsart}
\usepackage{amsmath, amsthm, amscd, amsfonts, amssymb, graphicx, color}
\usepackage[bookmarksnumbered, colorlinks, plainpages]{hyperref}
\hypersetup{colorlinks=true,linkcolor=red, anchorcolor=green, citecolor=cyan, urlcolor=red, filecolor=magenta, pdftoolbar=true}

\newtheorem{theorem}{Theorem}[section]
\newtheorem{lemma}[theorem]{Lemma}
\newtheorem{proposition}[theorem]{Proposition}
\newtheorem{cor}[theorem]{Corollary}

\newtheorem{question}[theorem]{Question}

\theoremstyle{remark}
\newtheorem{remark}[theorem]{\bf{Remark}}
\numberwithin{equation}{section}
\usepackage{multirow}
\usepackage{makecell}
\begin{document}
	\title [Weighted composition operators]{Weighted composition operators on weighted Fock spaces} 
	
	\author[S. Barik, A. Sen and K. Paul]{Somdatta Barik, Anirban Sen and Kallol paul}
	\address[Barik]{Department of Mathematics, Jadavpur University, Kolkata 700032, West Bengal, India}
	\email{bariksomdatta97@gmail.com}

	\address[Sen] {Mathematical Institute, Silesian University in Opava, Na Rybn\'{\i}\v{c}ku 1, 74601 Opava, Czech Republic}
	\email{anirbansenfulia@gmail.com; Anirban.Sen@math.slu.cz}

	\address[Paul] {Vice-Chancellor\\
		Kalyani University\\
		West Bengal 741235 \\and 
		Professor (on lien)\\ Department of mathematics\\ Jadavpur University\\Kolkata 700032\\West Bengal\\India}
	\email{kalloldada@gmail.com}
	
	\subjclass[2020]{Primary: 47B38; Secondary: 47B33, 30D15}
	
	\keywords{Weighted Fock spaces, weighted composition operators, boundedness, compactness, essential norm, Hilbert-Schmidt operator}
	\begin{abstract}
		We provide a complete characterization of the bounded, compact, and Hilbert–Schmidt class weighted composition operators acting on weighted Fock spaces, extending the classical Fock space characterization due to Le \cite{LBLMS_14}. An estimate for the essential norm of these operators is also provided. Our approach relies on the asymptotic behavior of the Mittag-Leffler function. As applications, we establish explicit criteria for weighted composition operators with exponential weights and recover corresponding results for composition operators.
	\end{abstract}
	\maketitle	
	\tableofcontents

	\section{Introduction}

	Let $\Omega$ be a domain, $\psi$ a holomorphic function on $\Omega$ and $\phi$ a holomorphic self-map of $\Omega.$ The weighted composition operator on the space of all holomorphic functions on $\Omega$ is defined by 
	$$C_{\psi, \phi}f=\psi \cdot (f \circ \phi).$$ 
	Such operator can be viewed as products of multiplication operator $M_{\psi}$ and composition operator $C_{\phi}$. Composition and weighted composition operators have attracted considerable attention in recent years and have been studied extensively on various spaces of holomorphic functions (see, e.g., \cite{BCM_ACM_03,CGP_MA_2015,CR_JLMS_2004,G_PAMS_2008,T_TAMS_2003}). A major objective is to relate the analytic properties of the symbols $\psi$ and $\phi$ to the operator-theoretic properties of $C_{\psi,\phi},$ such as boundedness, compactness. 
	
	In the context of Fock spaces, Ueki \cite{Ueki_P_07} characterized the boundedness and compactness of weighted composition operators. He proved that $C_{\psi, \phi}$ is bounded on $F^2(\mathbb C)$ if and only if the
	integral transform $ B_\phi(|\psi|^2
	)$ is bounded. Likewise, $C_{\psi, \phi}$ is compact if and only if  $\lim\limits_{|z|\to \infty} \sup B_\phi(|\psi|^2
	)(z)=0.$ While this criterion is theoretically satisfactory, it is difficult to check in practice. Subsequently, Le \cite{LBLMS_14} derived a more explicit criterion by proving that $C_{\psi, \phi}$ is bounded if and only if $\psi \in F^2(\mathbb C),$ $\phi(z)=\phi(0)+\lambda z$ with $|\lambda|\leq 1,$ and $$\sup\limits_{z\in\mathbb C}{|\psi(z)|^2e^{|\phi(z)|^2-|z|^2}}<\infty.$$ Furthermore, he established that $C_{\psi, \phi}$ is compact if and only if $\psi \in F^2(\mathbb C),$ $\phi(z)=\phi(0)+\lambda z$ with $|\lambda|< 1,$ and $$\lim\limits_{|z|\to\infty}|\psi(z)|^2 e^{\alpha(|\phi(z)|^2-|z|^2)}= 0.$$
	For further results concerning weighted composition operators on different Fock spaces, including their boundedness, compactness, Hilbert--Schmidt property, and other operator-theoretic characterizations, we refer the reader to \cite{CG_JMAA_21, L_JOT_17,TK_PA_19,T_AM_2013, U_IJMA_07, WCZ_IEOP_13}.

	Herein, we focus our study on weighted composition operators acting on weighted Fock spaces $F^2_{\alpha, m}(\mathbb C),$ a class of spaces recently introduced in \cite{BFY_AM_2026}.
	For $m>0$ and $\alpha>0,$ the weighted Fock space $F^2_{\alpha, m}(\mathbb C)$ consists of all entire functions on $\mathbb C$ that are square integrable with respect to the measure 
	$$d\mu_{\alpha, m}(z)=\frac{m\alpha^{\frac 2m}}{2\pi \Gamma(\frac 2m)}e^{-\alpha |z|^m} dA(z),$$
	where $dA$ is the Lebesgue measure on $\mathbb C$ and $\Gamma$ is the Euler gamma function. We note that in the particular case $m=2$ and $\alpha=1$ this space reduces to the classical Fock space $F^2(\mathbb C),$ also known as classical Segal-Bargmann space. 
	%Suppose $m>0$ and for each $\alpha>0,$ we consider the measure on the complex plane $\mathbb C$ is $$d\mu_{\alpha, m}(z)=\frac{m\alpha^{\frac 2m}}{2\pi \Gamma(\frac 2m)}e^{-\alpha |z|^m} dA(z),$$ where $dA$ is the Lebesgue measure on the complex plane $\mathbb C$ and $\Gamma$ is the Euler gamma function. The weighted Fock space $F^2_{\alpha, m}(\mathbb C),$ consists of all entire functions on $\mathbb C$ that are square integrable with respect to the measure $\mu_{\alpha, m}$. 
	The inner product on $F^2_{\alpha, m}(\mathbb C)$ is given by 
	$$ \langle f,g \rangle_{\alpha,m}=\int_{\mathbb C}f(z)\overline{g(z)}d\mu_{\alpha, m}(z)=\frac{m\alpha^{\frac 2m}}{2\pi \Gamma(\frac 2m)}\int_{\mathbb C}f(z)\overline{g(z)}e^{-\alpha |z|^m} dA(z).$$
	We denote the corresponding norm by $\|\cdot\|_{\alpha,m}$. 
	A direct computation gives $$ \|z^n\|_{\alpha, m}=\sqrt{\frac{\Gamma(\frac{2n+2}{m})}{\alpha^{\frac{2n}{m}}\Gamma(\frac 2m)}}$$ and $\langle z^n,z^k\rangle_{\alpha, m}=0$ for $n\neq k.$ Consequently, the functions  
	$$e_n(z)=\sqrt{\frac{\alpha^{\frac{2n}{m}}\Gamma(\frac 2m)}{\Gamma{(\frac{2n+2}{m})}}}z^n,\,\,n=0,1,2,\ldots$$ 
	form an orthonormal set in $F^2_{\alpha, m}(\mathbb C).$
	Following an argument analogous to that for the classical Fock space, one easily verifies that this set is in fact an orthonormal basis for the space.
	It is well known that $F^2_{\alpha, m}(\mathbb C)$ is a reproducing kernel Hilbert space. Its reproducing kernel at the point $z$ is given by
	$$k^{\alpha,m}_z(w)= \Gamma\left({\frac 2m}\right)\sum\limits_{n=0}^{\infty}\frac{(\alpha^{\frac2m}z\overline{w})^n}{\Gamma(\frac{2n+2}{m})}.$$ From this expression, the norm of the kernel function can be computed as
	$$\|k^{\alpha,m}_z\|_{\alpha,m}=\sqrt{\Gamma\left({\frac 2m}\right)\sum\limits_{n=0}^{\infty}\frac{(\alpha^{\frac2m}|z|^2)^n}{\Gamma(\frac{2n+2}{m})}}.$$
	As with the classical Fock space $F^2(\mathbb C),$ a fundamental property of the weighted Fock space $F^2_{\alpha,m}(\mathbb C)$  is the pointwise decay of its elements: every function $f\in F^2_{\alpha,m}(\mathbb C)$ satisfies
	\begin{eqnarray}\label{eq_9898}
		f(z)e^{-\frac{\alpha}{2}|z|^m}\to 0\,\,\text{as}\,\, |z|\to\infty.
	\end{eqnarray}
	This decay property plays a crucial role in the function theory of these spaces and mirrors the well-known behavior of the classical Fock space.
	For a comprehensive treatment of the Fock space over $\mathbb C,$ we refer the reader to the monograph \cite{Zhu_Fock_Book}.

	Our main goal is to characterize the bounded and compact weighted composition operators on the weighted Fock spaces $F^2_{\alpha, m}(\mathbb C)$, a setting in which, for $m=2$ and $\alpha=1$ our characterization recovers the classical Fock space result established in \cite{LBLMS_14}. However, the techniques employed in that classical setting are no longer applicable here, and new tools are required to obtain the corresponding results in this more general framework. Chief among these is the Mittag-Leffler function, whose asymptotic behaviour plays a central role in our proofs.
	Recall that the two-parameter Mittag–Leffler function is defined by
	\[
	E_{\alpha,\beta}(z):=\sum_{n=0}^{\infty}
	\frac{z^n}{\Gamma(\alpha n+\beta)},
	\qquad \alpha>0,\ \beta\in\mathbb{R},
	\] 
	(see \cite{HMS_11}).
	According to \cite{HMS_11}, this function admits the following asymptotic expansion. Let $\alpha \in (0,2)$, $\beta \in \mathbb R,$ and set $N^\ast=\lfloor \frac{\beta}{\alpha} \rfloor.$ If $\frac{\pi \alpha}{2} <\mu <\min\{\pi,\pi\alpha\}$ then 
	\begin{eqnarray}\label{EX1}
		E_{\alpha,\beta}(z)=\frac{1}{\alpha}z^{\frac{(1-\beta)}{\alpha}}e^{z^{\frac{1}{\alpha}}}-\sum\limits_{r=1}^{N^\ast}\frac{1}{\Gamma(\beta-\alpha r)}\frac{1}{z^r}+O\left[\frac{1}{z^{N^\ast+1}}\right],
	\end{eqnarray}
	as $|z|\to\infty,$ $|\text{arg} z|\leq\mu.$ For two nonnegative quantities $A$ and $B,$ we write $A \lesssim B$ if there exists a constant $C>0$ independent of the relevant variable, such that $A \leq CB,$ we write $A \asymp B$ if both $A \lesssim B$ and $B \lesssim A$ hold.
	From the expansion \eqref{EX1}, it is easy to see that for real $x>0$
	\begin{eqnarray}\label{neweq2}
		E_{\alpha,\beta}(x)\asymp \frac{1}{\alpha}x^{\frac{(1-\beta)}{\alpha}}e^{x^{\frac{1}{\alpha}}},~~~~\,\,\, x \to \infty.
	\end{eqnarray}
	Consequently, for any $\epsilon \in (0,1),$ there exists $x_0>0$ such that
	\[
	E_{\alpha,\beta}(x)\ge e^{(1-\epsilon)x^{1/\alpha}},
	\qquad x\ge x_0.
	\]
	Since $E_{\alpha,\beta}$ is continuous and positive on the compact interval
	$[0,x_0]$, the function
	\[
	x\mapsto E_{\alpha,\beta}(x)e^{-(1-\epsilon)x^{1/\alpha}}
	\]
	attains its minimum value there. Thus, defining
	\[
	C_{\epsilon}=\min\left\{1,\min_{ x\in [0,x_0]}
	E_{\alpha,\beta}(x)e^{-(1-\epsilon)x^{1/\alpha}}\right\}>0,
	\]
	we obtain the global lower bound
	\begin{eqnarray}\label{neweq1}
		E_{\alpha,\beta}(x)\ge C_{\epsilon} e^{(1-\epsilon)x^{1/\alpha}},
		\qquad x\ge0.
	\end{eqnarray}

	We now present the principal result of this paper. We begin by establishing a characterization of bounded weighted composition operators on the weighted Fock spaces $F^2_{\alpha, m}(\mathbb C).$
	
	\begin{theorem}\label{Th_1}
		Let $\psi\in F^2_{\alpha, m}(\mathbb C)$ and $\phi$ be entire function on $\mathbb C$ such that $\psi$ is not identically zero. Then the following are equivalent:
		
		\begin{itemize}
			\item [(i)] $C_{\psi,\phi}$ is bounded on $F^2_{\alpha, m}(\mathbb C).$
			\item[(ii)] $\phi(z)=\phi(0)+\lambda z$ with $|\lambda|\leq 1$ and $\,\sup \{|\psi(z)|^2 e^{\alpha(|\phi(z)|^m-|z|^m)}: z\in\mathbb C\}<\infty.$
		\end{itemize}
	\end{theorem}

	Our next result provides a characterization of compact weighted composition operators on $F^2_{\alpha, m}(\mathbb C).$
	\begin{theorem}\label{Th_2}
		Let $\psi\in F^2_{\alpha, m}(\mathbb C)$ and $\phi$ be entire function on $\mathbb C$ such that $\psi$ is not identically zero. Then the following are equivalent:
		\begin{itemize}
			\item [(i)]  $C_{\psi,\phi}$ is compact on $F^2_{\alpha, m}(\mathbb C).$
			\item [(ii)] $\phi(z)=\phi(0)+\lambda z$ for some $|\lambda|< 1$ and $ \lim\limits_{|z|\to\infty}|\psi(z)|^2 e^{\alpha(|\phi(z)|^m-|z|^m)}= 0.$
		\end{itemize}
	\end{theorem}
	
	We then turn our attention to the essential norm of bounded weighted composition operators. To this end, we first recall the definition of the essential norm.
	If $T$ is a bounded linear operator on 
	$F^2_{\alpha,m}(\mathbb C)$ then the essential norm of $T$ is defined by
	\[
	\|T\|_{e}
	=
	\inf\bigl\{\|T-K\|: K \text{ is a compact operator on } F^2_{\alpha,m}\bigr\}.
	\]
	The following theorem provides an asymptotic characterization of the essential norm.

	\begin{theorem}\label{Th_3}
		Let $C_{\psi,\phi}$ be bounded on $F^2_{\alpha, m}(\mathbb C)$ and $\phi$ be non-constant. Then 
		$$ \|C_{\psi,\phi}\|_{e}^2\asymp  \lim_{|z|\to\infty}\sup |\psi(z)|^2e^{\alpha(|\phi(z)|^m-|z|^m)}.$$
	\end{theorem}

	Finally, we characterize the Hilbert--Schmidt class of weighted composition operators on $F^2_{\alpha, m}(\mathbb C).$ The next theorem gives a necessary and sufficient condition in terms of an explicit integral involving the symbols $\psi$ and $\phi.$
	\begin{theorem}\label{Th_4}
		Let $\psi\in F^2_{\alpha, m}(\mathbb C)$ and $\phi$ be entire function on $\mathbb C.$ Then $C_{\psi,\phi}$ is Hilbert-Schmidt on $F^2_{\alpha, m}(\mathbb C)$ if and only if $$\int_{\mathbb C} |\psi(z)|^2 |\phi(z)|^{m-2} e^{\alpha ({|\phi(z)|^m-|z|^m})}dA(z)<\infty.$$
	\end{theorem}

	This article is organized as follows. Section~\ref{S1} is devoted to the characterization of bounded weighted composition operators on $F^2_{\alpha,m}(\mathbb C)$. In Section~ \ref{S2}, we characterize compact weighted composition operators. In Section~\ref{S3}, we obtain an asymptotic characterization of the essential norm of bounded weighted composition operators and establish a necessary and sufficient condition for such operators to belong to the Hilbert-Schmidt class. Section~\ref{S4} applies the main results, presenting illustrative examples, explicit boundedness and compactness criteria for weighted composition operators with exponential weights, and corresponding characterizations for composition operators. Finally, Section~\ref{S5} provides examples distinguishing bounded, compact, and Hilbert–Schmidt weighted composition operators, and Section~\ref{S6} contains final remarks.
	\section {Boundedness}\label{S1}
	We begin by establishing a complete characterization of bounded weighted composition operators $C_{\psi,\phi}$ on $F^2_{\alpha, m}(\mathbb C).$
	\begin{proof}[Proof of Theorem \ref{Th_1}]
		First we prove (ii) implies (i). We consider the following two cases separately.\\
		Case 1: Suppose $\lambda=0.$ Then $\phi(z)=\phi(0)$ for all $z$.  Thus, $C_{\psi,\phi}$ is a rank one operator and hence bounded. \\
		Case 2: Let $\lambda\neq 0.$ For $f\in F^2_{\alpha, m}(\mathbb C),$ we get
		\begin{eqnarray*}
			\|C_{\psi,\phi}f\|_{\alpha,m}^2&=&\int_{\mathbb C}|\psi(z)|^2|f(\phi(z))|^2d\mu_{\alpha, m}(z)\\
			&=&\int_{\mathbb C}|\psi(z)|^2|f(\phi(z))|^2\frac{m\alpha^{\frac 2m}}{2\pi \Gamma(\frac 2m)}e^{-\alpha |z|^m} dA(z)\\
			&=& \int_{\mathbb C}|\psi(z)|^2|f(\phi(z))|^2\frac{m\alpha^{\frac 2m}}{2\pi \Gamma(\frac 2m)}e^{-\alpha |\phi(z)|^m}e^{\alpha (|\phi(z)|^m-|z|^m)} dA(z)\\
			&\leq&M_{\alpha, m}(\psi,\phi) \int_{\mathbb C}|f(\phi(z))|^2\frac{m\alpha^{\frac 2m}}{2\pi \Gamma(\frac 2m)}e^{-\alpha |\phi(z)|^m}dA(z),
		\end{eqnarray*}
		where $M_{\alpha, m}(\psi,\phi)= \sup \{|\psi(z)|^2 e^{\alpha(|\phi(z)|^m-|z|^m)}: z\in\mathbb C\}.$
		By the change of variable $w=\phi(z)=\phi(0)+\lambda z,$ we get 
		\begin{eqnarray*}
			\|C_{\psi,\phi}f\|_{\alpha,m}^2&\leq& |\lambda|^{-2}M_{\alpha, m}(\psi,\phi) \int_{\mathbb C}|f(w)|^2e^{-\alpha |w|^m}\frac{m\alpha^{\frac 2m}}{2\pi \Gamma(\frac 2m)}dA(w)\\
			&=& |\lambda|^{-2}M_{\alpha, m}(\psi,\phi) \int_{\mathbb C}|f(w)|^2d\mu_{\alpha, m}(w)\\
			&=&|\lambda|^{-2}M_{\alpha, m}(\psi,\phi)\|f\|^2_{\alpha,m}.
		\end{eqnarray*}
		It follows that $C_{\psi,\phi}$ is bounded on $F^2_{\alpha, m}(\mathbb C).$\\

		Next, we prove (i) implies (ii).  For every $z\in\mathbb C,$ a direct application of the reproducing property yields $C_{\psi,\phi}^\ast k_z^{\alpha,m}=\overline{\psi(z)}k_{\phi(z)}^{\alpha,m}.$ Now, 
		\begin{eqnarray*}
			\|C_{\psi,\phi}^\ast k_z^{\alpha,m}\|_{\alpha,m}^2
			%&=&\|\overline{\psi(z)}k_{\phi(z)}^{\alpha,m}\|_{\alpha,m}^2\\
			&=&|\psi(z)|^2\|k_{\phi(z)}^{\alpha,m}\|_{\alpha,m}^2\\
			&=&|\psi(z)|^2\sum\limits_{n=0}^\infty \frac{\Gamma(\frac 2m)}{\Gamma(\frac{2n+2}{m})}(\alpha^{\frac 2m}|\phi(z)|^2)^n.
		\end{eqnarray*}
		Thus, we have
		\begin{eqnarray}\label{eqlb2}
			\frac{\|C_{\psi,\phi}^\ast k_z^{\alpha,m}\|_{\alpha,m}^2}{\|k_z^{\alpha,m}\|_{\alpha,m}^2}
			&=&|\psi(z)|^2\frac{\sum\limits_{n=0}^\infty \frac{\Gamma(\frac 2m)}{\Gamma(\frac{2n+2}{m})}(\alpha^{\frac 2m}|\phi(z)|^2)^n}{\sum\limits_{n=0}^\infty \frac{\Gamma(\frac 2m)}{\Gamma(\frac{2n+2}{m})}(\alpha^{\frac 2m}|z|^2)^n}\nonumber\\
			&=&|\psi(z)|^2\frac{E_{\frac {2}{m},\frac {2}{m}}(\alpha^{\frac 2m} |\phi(z)|^2)}{E_{\frac {2}{m},\frac {2}{m}}(\alpha^{\frac 2m} |z|^2)}.
		\end{eqnarray}
		Since $C_{\psi,\phi}$ is bounded on $F^2_{\alpha, m}(\mathbb C)$, there exists $M>0$ such that 
		\begin{eqnarray}\label{Feq_1}
			|\psi(z)|^2\frac{E_{\frac {2}{m},\frac {2}{m}}(\alpha^{\frac 2m} |\phi(z)|^2)}{E_{\frac {2}{m},\frac {2}{m}}(\alpha^{\frac 2m} |z|^2)}<M\quad\quad \forall~z\in\mathbb C.
		\end{eqnarray}
		Let $0<\epsilon<1.$ It follows from \eqref{neweq1} that there exists $C_{\epsilon}>0$ such that 
		\begin{eqnarray}\label{eqcomp1}
			E_{\frac {2}{m},\frac {2}{m}}(\alpha^{\frac 2m} |\phi(z)|^2)\geq C_{\epsilon}e^{\alpha(1-\epsilon)|\phi(z)|^m}.
		\end{eqnarray}
		Now, from \eqref{neweq2}, as $|z|\to\infty,$ 
		\begin{eqnarray}\label{eqcomp2}
			E_{\frac {2}{m},\frac {2}{m}}(\alpha^{\frac 2m} |z|^2)\asymp |z|^{m-2}e^{\alpha|z|^m}.
		\end{eqnarray}
		Using \eqref{eqcomp1} and \eqref{eqcomp2} in \eqref{Feq_1}, we conclude that there exists $M_{\epsilon}'>0$ such that
		\begin{eqnarray}\label{eqlb}
			|\psi(z)|^2 |z|^{2-m}e^{\alpha((1-\epsilon)|\phi(z)|^m-|z|^m)}<M_{\epsilon}'
		\end{eqnarray} as $|z|\to\infty.$
		Since $\psi$ is not identically $0,$ there exist an integer $k\geq0$ and an entire function $h$ with $h(0)\neq0$ such that $\psi(z)=z^k h(z).$ From \eqref{eqlb}, we get
		\begin{eqnarray*}
			|z^k h(z)|^2 |z|^{2-m}e^{\alpha((1-\epsilon)|\phi(z)|^m-|z|^m)}<M_{\epsilon}'.
		\end{eqnarray*}
		Applying the logarithm to both sides yields
		$$(2k-m+2)\log|z|+2\log|h(z)|+\alpha\left((1-\epsilon)|\phi(z)|^m-|z|^m\right)<\log M_{\epsilon}'.$$ 
		For any $R>0,$ setting $z=R e^{i\theta}$ and integrating with respect to $\theta$ over $[-\pi,\pi],$ we obtain 
		\begin{eqnarray*}
			(2k-m+2)\log R+\frac{1}{\pi}\int_{-\pi}^\pi\log|h(R e^{i\theta})|d\theta
			+\frac{\alpha}{2\pi}\int_{-\pi}^\pi\left((1-\epsilon)|\phi(R e^{i\theta})|^m-R^m\right)d\theta<\log M_{\epsilon}'.
		\end{eqnarray*}
		Jensen's inequality gives
		$$\frac{1}{2\pi}\int_{-\pi}^{\pi}\log|h(R e^{i\theta})| d\theta\geq \log|h(0)|.$$ It follows that 
		\begin{eqnarray}\label{eqf11}
			(2k-m+2)\log R+2\log|h(0)|-\alpha R^m
			+\frac{\alpha (1-\epsilon)}{2\pi}\int_{-\pi}^\pi|\phi(R e^{i\theta})|^m d\theta<\log M_{\epsilon}'.
		\end{eqnarray}
		
		Now, we claim that $\phi$ is polynomial. Suppose, on the contrary, that $\phi$ is a transcendental function. 
		Then it follows from \cite[Th. 1.1]{BFY_14} that $$\lim\limits_{R\to\infty}\frac{\log M(\phi, R)}{\log R}=\infty,$$
		where $M(\phi, R)=\max\limits_{|z|=R}|\phi(z)|.$ Thus, for any $N>0$ there exists $R_0>1$ such that
		\begin{eqnarray}\label{eqFF1}
			&&\frac{\log M(\phi, R)}{\log R}>N
			\implies M(\phi, R)>R^N~~~~ \text{for all}~~ R>R_0.
		\end{eqnarray} Consider that 
		$$M_m(\phi, R)=\left(\frac{1}{2\pi}\int_{-\pi}^\pi|\phi(R e^{i\theta})|^m d\theta\right)^{\frac 1m}.$$ Since $|\phi|^m$ is subharmonic, for $0<r<R,$
		$$|\phi(r e^{it})|^m\leq \frac{R+r}{R-r}\frac{1}{2\pi}\int_{-\pi}^\pi|\phi(R e^{i\theta})|^m d\theta.$$ Taking supremum over all $t\in[-\pi,\pi],$ we get 
		$M(\phi, r)\leq \left(\frac{R+r}{R-r}\right)^{\frac 1m} M_m(\phi, R)$ for $0<r<R.$ In particular, consider $r=\frac R2.$ Then 
		\begin{eqnarray}\label{eqff2}
			M\left(\phi,\frac R2\right)\leq 3^{\frac 1m} M_m(\phi, R). 
		\end{eqnarray}
		From \eqref{eqFF1} and \eqref{eqff2}, it follows that for every $N>0,$ 
		\begin{eqnarray}\label{eq_F22}
			M_m(\phi, R)>\frac{1}{2^N 3^{\frac 1m}}R^N,\quad\quad\text{whenever}~~ R\geq2R_0.
		\end{eqnarray}   
		Using \eqref{eq_F22} in \eqref{eqf11}, we get 
		\begin{align}\label{eqf112}
			&(2k-m+2)\log R+2\log|h(0)|-\alpha R^m+\frac{\alpha (1-\epsilon)}{3\times2^{mN}} R^{mN}<\log M_{\epsilon}'\nonumber\\
			&\implies \frac{\alpha}{3}\left(\frac{1-\epsilon}{2^{mN}}R^{m(N-1)}-3\right)<\frac{1}{R^m}\left((m-2k-2)\log R-2\log|h(0)|\right).
		\end{align}
		If $N>1$ then the left hand side of \eqref{eqf112} becomes unbounded as $R\to\infty$, while the right hand side tends to $0,$ which is impossible. Thus, $\phi$ can not be a transcendental function. Therefore, $\phi$ must be polynomial. Suppose $\phi(z)=\sum\limits_{n=0}^pa_n z^n.$ We next show that $\phi$ is of degree at most one. Then 
		$$|\phi(Re^{i\theta})|=R^p \left|a_pe^{ip\theta}+\sum\limits_{n=0}^{p-1}a_n R^{n-p} e^{in\theta}\right|.$$ Since $\lim\limits_{R\to\infty}\frac{ |\phi(Re^{i\theta})|}{|a_p|R^p}=1,$ for every $0<\delta<1$ there exists $R_1>0$ such that 
		\begin{eqnarray}\label{eqbound998}
			(1-\delta)|a_p|R^p\leq|\phi(Re^{i\theta})|,\quad\quad \text{whenever}~R\geq R_1.
		\end{eqnarray} By using \eqref{eqbound998} in \eqref{eqf11},
		we get
		\begin{align}\label{eqbound22}
			&(2k-m+2)\log R+2\log|h(0)|
			+\alpha R^m((1-\epsilon)(1-\delta)^m|a_p|^mR^{m(p-1)}-1)<\log M_{\epsilon}'\nonumber\\
			&\implies \alpha ((1-\epsilon)(1-\delta)^m|a_p|^mR^{m(p-1)}-1)\leq \frac{1}{R^m} \left( (m-2k-2)\log R-2\log|h(0)|\right).
		\end{align}
		If $p> 1$, then
		\[
		R^{m(p-1)}\longrightarrow\infty \quad \text{as } R\to\infty.
		\]
		Hence,
		\[
		\alpha \left((1-\epsilon)(1-\delta)^m|a_p|^mR^{m(p-1)}-1\right)\longrightarrow\infty,
		\]
		which implies that the left hand side of \eqref{eqbound22} is unbounded as $R\to\infty,$ while the right hand side tends to $0$. This contradicts \eqref{eqbound22}. Therefore, $p\leq1$. Thus, $\phi(z)=a_0+a_1z.$ Clearly, $a_0=\phi(0)$ and $a_1=\lambda(\text{say}).$ If $p=1$, then \eqref{eqbound22} becomes
		\begin{eqnarray*}
			&&(2k-m+2)\log R+2\log|h(0)|
			+\alpha R^m\left((1-\epsilon)(1-\delta)^m|\lambda|^m-1\right)
			<\log M_{\epsilon}'\\
			&&\implies \alpha \left((1-\epsilon)(1-\delta)^m|\lambda|^m-1\right)\leq \frac {1}{R^m} (m-2k-2)\log R-2\log|h(0)|.
		\end{eqnarray*}
		Since $R^m\to\infty$ as $R\to\infty$, the above inequality can hold only if
		\[
		(1-\epsilon)(1-\delta)^m|\lambda|^m-1\leq0,
		\]
		that is,
		\[
		|\lambda|\leq\frac{1}{(1-\epsilon)^{\frac 1m}(1-\delta)}.
		\]
		Since $\epsilon$ and $\delta$ are arbitrary positive numbers, then we conclude that $|\lambda|\leq1.$
		Therefore, $
		\phi(z)=\phi(0)+\lambda z,~|\lambda|\leq1.$ 
		Finally, we show that 
		$$\sup_{z\in\mathbb C} \{|\psi(z)|^2 e^{\alpha(|\phi(z)|^m-|z|^m)}: z\in\mathbb C\}<\infty.$$
		Note that, for $|\lambda|\in(0,1],~|\phi(z)|=|\phi(0)+\lambda z|\to\infty$ as $|z|\to\infty.$ Then by applying the asymptotic expansion we have 
		\begin{eqnarray}\label{eqon45}
			\frac{E_{\frac {2}{m},\frac {2}{m}}(\alpha^{\frac 2m} |\phi(z)|^2)}{E_{\frac {2}{m},\frac {2}{m}}(\alpha^{\frac 2m} |z|^2)}\asymp \left|\frac{\phi(z)}{z}\right|^{m-2}e^{\alpha(|\phi(z)|^m-|z|^m)}\quad\quad \text{as}~|z|\to\infty.
		\end{eqnarray}
		From \eqref{Feq_1}, there exists $K'>0$ such that 
		\begin{eqnarray}\label{eqlab_01}
			|\psi(z)|^2\left|\frac{\phi(z)}{z}\right|^{m-2}e^{\alpha(|\phi(z)|^m-|z|^m)}<K',
		\end{eqnarray}
		as $|z|\to\infty.$
		Then it follows from \eqref{eqlab_01} that for any $|\lambda|\in(0,1]$ there exists $K>0$ such that $$ |\psi(z)|^2 e^{\alpha(|\phi(z)|^m-|z|^m)}<K$$ as $|z|\to\infty,$ and so there exists $R_2>0$ such that  $$ |\psi(z)|^2 e^{\alpha(|\phi(z)|^m-|z|^m)}<K \text{ for all }|z|>R_2.$$ Since $\{z:|z|\leq R_2\}$ is compact, the function $|\psi(z)|^2 e^{\alpha(|\phi(z)|^m-|z|^m)}$ is also bounded on $\{z:|z|\leq R_2\}$. Therefore, $$\sup_{z\in\mathbb C} \{|\psi(z)|^2 e^{\alpha(|\phi(z)|^m-|z|^m)}: z\in\mathbb C\}<\infty.$$
		Since $\psi\in F^2_{\alpha,m}(\mathbb C),$ for the case $\lambda=0,$ it follows from \eqref{eq_9898} that 
		$$|\psi(z)|^2 e^{\alpha(|\phi(0)|^m-|z|^m)}\to0~~\,\,\,\,\text{as}~ |z|\to\infty,$$ and hence the result follows. 
	\end{proof}

	The next result is a useful consequence of the preceding discussion.
	
	\begin{cor}\label{Rm1}
		Let $\phi(z)=\phi(0)+\lambda z$ with $|\lambda|=1$ and  suppose that $\psi$ has a zero of order at least one. Then $C_{\psi,\phi}$ is not bounded on $F^2_{\alpha, m}(\mathbb C).$
	\end{cor}
	\begin{proof}
		Assume, to the contrary, that $C_{\psi,\phi}$ is bounded.
		Then by Theorem \ref{Th_1}, there exists $M>0$ such that 
		$$|\psi(z)|^2 e^{\alpha(|\phi(z)|^m-|z|^m)}<M,~~\forall ~z\in\mathbb C.$$ Taking logarithm both side, we get 
		$$2\log|\psi(z)|+\alpha(|\phi(z)|^m-|z|^m)<\log M.$$
		Writing $\psi(z)=z^kh(z),$ where $k\geq1$ and $h(0)\neq0,$ taking logarithms, we obtain
		$$2k\log|z|+2\log|h(z)|+\alpha(|\phi(z)|^m-|z|^m)<\log M.$$
		For any $R>0,$ putting $z=R e^{i\theta}$ and integrating with respect to $\theta$ on $[-\pi,\pi]$ yields
		\begin{eqnarray}\label{NR1}
			2k\log R+2\log |h(0)|+\frac{\alpha}{2\pi}\int_{-\pi}^{\pi}(|\phi(Re^{i\theta})|^m-R^m) d\theta<\log M.
		\end{eqnarray}
		Since $|\lambda|=1,$ we have $|\phi(Re^{i\theta})|^m-R^m=mR^m\Re\left(\frac{\phi(0)}{\lambda Re^{i\theta}}\right)+O(R^{m-2}),$ where $\Re(\cdot)$ is the real part. 
		Since \[\int_{-\pi}^\pi\Re\left(\frac{\phi(0)}{\lambda Re^{i\theta}}\right)d\theta=0,\] it follows from \eqref{NR1} that for $0<m\leq2$,
		\begin{eqnarray}\label{eq_comp_41}
			2k\log R+2\log |h(0)|+O(R^{m-2}) <\log M.
		\end{eqnarray}
		As $R\to\infty,$ left hand side of \eqref{eq_comp_41} cannot hold, yielding a contradiction. \\
		On the other hand, if $m>2,$ Jensen's inequality gives
		\begin{eqnarray*}
			\frac{1}{2\pi}\int_{-\pi}^{\pi}|\phi(Re^{i\theta})|^md\theta
			&\geq&\left(\frac{1}{2\pi}\int_{-\pi}^{\pi}|\phi(Re^{i\theta})|^2  d\theta\right)^{\frac m2}\\
			&=&(|\phi(0)|^2+R^2)^{\frac m2}
			\geq R^m.
		\end{eqnarray*}
		It follows from \eqref{NR1} that 
		\begin{eqnarray}\label{eq6571}
			2k\log R+2\log |h(0)|<\log M,
		\end{eqnarray}
		which again leads to a contradiction as $R\to\infty.$ Therefore, $C_{\psi,\phi}$ is not bounded on $F^2_{\alpha, m}(\mathbb C).$\\
	\end{proof}
	
	\begin{remark}
		In contrast, when $|\lambda|<1,$ we construct a bounded weighted composition operator $C_{\psi,\phi}$ for which $\psi$ has a zero of order at least $1.$ Indeed, let $\psi(z)=z^k,$ where $k\geq 1$ and $\phi(z)=\lambda z.$ A direct calculation shows $$\sup_{z \in \mathbb C}|\psi(z)|^2 e^{\alpha(|\phi(0)|^m-|z|^m)}=\sup_{z \in \mathbb C}|z|^{2k} e^{-\alpha(1-|\lambda|^m)|z|^m}<\infty.$$
		Thus, the weighted composition operator $C_{\psi,\phi}$ is bounded on $F^2_{\alpha, m}(\mathbb C).$
	\end{remark}

	\section {Compactness}\label{S2}
	The main objective of this section is to characterize compact weighted composition operators on $F^2_{\alpha, m}(\mathbb C).$ We begin by proving the following characterization theorem.

	\begin{proof}[Proof of Theorem \ref{Th_2}]
		We first establish that (i) implies (ii).
		Suppose that $C_{\psi,\phi}$ is compact on $F^2_{\alpha, m}(\mathbb C).$ Then Theorem \ref{Th_1} implies that $\phi(z)=\phi(0)+\lambda z,~|\lambda|\leq1.$ By the asymptotic expansion 
		$$\|k^{\alpha,m}_z\|^2_{\alpha,m}= E_{\frac {2}{m},\frac {2}{m}}(\alpha^{\frac 2m} |z|^2)\asymp |z|^{m-2}e^{\alpha|z|^m}.$$
		It is easy to observe that
		the normalized kernels $\{\hat k^{\alpha,m}_z\}$ converge to $0$ weakly as $|z|\to\infty.$ Since $C_{\psi,\phi}$ is compact, so is 
		$C_{\psi,\phi}^\ast$. Thus, $$\|C_{\psi,\phi}^\ast \hat k_z^{\alpha,m}\|_{\alpha,m}^2=|\psi(z)|^2\frac{E_{\frac {2}{m},\frac {2}{m}}(\alpha^{\frac 2m} |\phi(z)|^2)}{E_{\frac {2}{m},\frac {2}{m}}(\alpha^{\frac 2m} |z|^2)}\to 0\,\,\,\,\,\,\,\,\,\text{as}~ |z|\to\infty.$$  It follows from \eqref{eqon45} that 
		\begin{eqnarray*}
			\lim_{|z|\to\infty} |\psi(z)|^2\left|\frac{\phi(z)}{z}\right|^{m-2}e^{\alpha(|\phi(z)|^m-|z|^m)}=0
		\end{eqnarray*}
		Note that, $\lim\limits_{|z|\to\infty}\left|\frac{\phi(z)}{z}\right|^{m-2}=|\lambda|^{m-2}.$ Thus, for $|\lambda|\in(0,1],$ 
		$$\lim_{|z|\to\infty} |\psi(z)|^2e^{\alpha(|\phi(z)|^m-|z|^m)}=0.$$ Since $\psi\in F^2_{\alpha,m}(\mathbb C),$ for the case $\lambda=0,$ it follows from \eqref{eq_9898} that $$|\psi(z)|^2 e^{\alpha(|\phi(0)|^m-|z|^m)}\to0 \,\,\,\,\,\,\text{as}~|z|\to\infty.$$
		Let $\epsilon>0.$ Then there exists $R>0$ such that \begin{eqnarray}\label{eqcom1}
			|\psi(z)|^2e^{\alpha(|\phi(z)|^m-|z|^m)}<\epsilon
		\end{eqnarray}
		for all $|z|\geq R.$ Since $\psi$ is not identically $0,$ there is integer $k\geq0$ and an entire function $h$ with $h(0)\neq0$ such that $\psi(z)=z^k h(z).$ Thus, 
		\begin{eqnarray}\label{eqcomp_2}
			|z|^{2k}|h(z)|^2e^{\alpha(|\phi(z)|^m-|z|^m)}<\epsilon
		\end{eqnarray}
		for all $|z|\geq R.$ Taking logarithm both side we get
		\begin{eqnarray}\label{eqbbb45}
			2k\log |z|+2\log |h(z)|+\alpha (|\phi(z)|^m-|z|^m)<\log \epsilon.
		\end{eqnarray}
		Now, setting $z=R e^{i\theta}$ and integrating with respect to $\theta$ on $[-\pi,\pi],$ Jensen's formula gives
		\begin{eqnarray}\label{eqcomp3}
			2k\log R+2\log |h(0)|+\frac{\alpha}{2\pi}\int_{-\pi}^{\pi}(|\phi(Re^{i\theta})|^m-R^m) d\theta<\log \epsilon.
		\end{eqnarray}
		We claim that $|\lambda|<1.$ Suppose, to the contrary, that $|\lambda|=1.$ If $k\geq 1,$ then Corollary \ref{Rm1} implies that $C_{\psi,\phi}$ cannot be bounded, contradicting the compactness of $C_{\psi,\phi}.$ Hence, we only consider the case $k=0.$
		We have the following expression
		$$|\phi(Re^{i\theta})|^m-R^m=mR^m\Re\left(\frac{\phi(0)}{\lambda Re^{i\theta}}\right)+O(R^{m-2}).$$ We distinguish two cases.\\
		Case 1: Let $0<m\leq2.$ Since $\int_{-\pi}^\pi\Re\left(\frac{\phi(0)}{\lambda Re^{i\theta}}\right)d\theta=0,$ then \eqref{eqcomp3} reduces to
		\begin{eqnarray}\label{ewcnew}
			2\log|h(0)|+O(R^{m-2})<\log\epsilon.
		\end{eqnarray}
		If $0<m<2$ then
		\[
		2\log|h(0)|+O(R^{m-2})\longrightarrow 2\log|h(0)|\quad \text{as}~R\to\infty.
		\]
		Since $h(0)\neq0$, choose $0<\epsilon<|h(0)|^2$. Then
		\[
		\log\epsilon<2\log|h(0)|.
		\]
		As the left hand side of \eqref{ewcnew} converges to the value $2\log|h(0)|$, which is strictly greater than $\log\epsilon$, it follows from the definition of the limit that there exists $R_0>0$ such that
		\[
		2\log|h(0)|+O(R^{m-2})>\log\epsilon
		\]
		for all $R\geq R_0$. This contradicts \eqref{ewcnew}. 
		If $m=2$ then equation \eqref{eqcomp3} reduces to
		\[
		2\log|h(0)|+\alpha|\phi(0)|^2<\log\epsilon.
		\]
		Since $h(0)\neq0$, choose
		\[
		0<\epsilon<|h(0)|^2e^{\alpha|\phi(0)|^2}.
		\]
		Then
		\[
		\log\epsilon<
		2\log|h(0)|+\alpha|\phi(0)|^2,
		\]
		which is again a contradiction. Therefore, $|\lambda|<1$.\\
		Case 2:  Let $m>2.$ Then by Jensen's inequality,
		\begin{eqnarray*}
			\frac{1}{2\pi}\int_{-\pi}^{\pi}|\phi(Re^{i\theta})|^md\theta
			\geq R^m.
		\end{eqnarray*}
		It follows from \eqref{eqcomp3} that 
		\begin{eqnarray}\label{eq657}
			2\log |h(0)|<\log \epsilon.
		\end{eqnarray}
		Since $h(0)\neq0,$ choose $0<\epsilon<|h(0)|^2.$ Then $\log\epsilon<2\log|h(0)|,$ which contradicts the above inequality. Thus, $|\lambda|<1.$

		Next, we prove (ii) implies (i).
		If $\lambda=0$ then as in the proof of Theorem \ref{Th_1}, the operator $C_{\psi,\phi}$ has rank one, hence compact. Consider $\lambda\neq 0.$ Let $\{f_n\}$ be a sequence in $F^2_{\alpha, m}(\mathbb C)$ converging weakly to $0.$ It is easy to observe that $\{f_n\}$ is bounded in norm and converges to $0$ uniformly on compact subsets of $\mathbb C.$ We need to show that $\|C_{\psi,\phi}f_n\|_{\alpha,m}\to 0$ as $n\to \infty.$ For $w\in\mathbb C,$ define 
		$$F_{\alpha,m}(w)=|\lambda|^{-2}|\psi(\phi^{-1}(w))|^2 e^{\alpha(|w|^m-|\phi^{-1}(w)|^m)}.$$ Since $\lim\limits_{|w|\to\infty}|\phi^{-1}(w)|=\infty,$ it follows that $\lim\limits_{|w|\to\infty}F_{\alpha,m}(w)=0.$ Now, for any $R>0,$ we compute
		\begin{eqnarray*}
			\|C_{\psi,\phi}f_n\|_{\alpha,m}^2
			&=&\int_{\mathbb C}|\psi(z)|^2|f_n(\phi(z))|^2\frac{m\alpha^{\frac 2m}}{2\pi \Gamma(\frac 2m)}e^{-\alpha |z|^m} dA(z)\\
			&=&\frac{m\alpha^{\frac 2m}}{2\pi \Gamma(\frac 2m)}\int_{\mathbb C}|\lambda|^2F_{\alpha,m}(\phi(z))|f_n(\phi(z))|^2e^{-\alpha |\phi(z)|^m} dA(z).
		\end{eqnarray*}
		By the change of variables $w=\phi(z)=\phi(0)+\lambda z,$ we have
		\begin{eqnarray*}
			\|C_{\psi,\phi}f_n\|_{\alpha,m}^2&=&\frac{m\alpha^{\frac 2m}}{2\pi \Gamma(\frac 2m)}\int_{\mathbb C}F_{\alpha,m}(w)|f_n(w)|^2e^{-\alpha |w|^m} dA(w)\\
			&\leq&\|F_{\alpha,m}\|_{\infty}\int_{|w|<R}|f_n(w)|^2d\mu_{\alpha, m}(w)+(\sup_{|w|>R}F_{\alpha,m}(w))\|f_n\|^2_{\alpha, m},
		\end{eqnarray*}
		where $\|F_{\alpha,m}\|_{\infty}=\sup\limits_{w\in\mathbb C}F_{\alpha,m}(w).$
		If possible suppose that $\lim\limits_{n\to\infty}\|C_{\psi,\phi}f_n\|_{\alpha,m}= l>0.$ Since $\{f_n\}$ is bounded, there exists $K$ such that $\|f_n\|_{\alpha,m}<K$ for all $n\in\mathbb N.$ Consider $\epsilon=\frac{l^2}{4k^2}.$ As $\lim\limits_{|w|\to\infty}F_{\alpha,m}(w)=0,$ there exists $R>0$ such that $|F_{\alpha,m}(w)|<\frac{l^2}{4k^2}$ for all $|w|>R.$ Thus,
		\begin{eqnarray*}
			\|C_{\psi,\phi}f_n\|_{\alpha,m}^2
			&\leq& \|F_{\alpha,m}\|_{\infty}\int_{|w|<R}|f_n(w)|^2d\mu_{\alpha, m}(w)+\frac{l^2}{4k^2}\|f_n\|^2_{\alpha, m}\\
			&\leq&\|F_{\alpha,m}\|_{\infty}\int_{|w|<R}|f_n(w)|^2d\mu_{\alpha, m}(w)+\frac{l^2}{4}.
		\end{eqnarray*}
		When $n\to\infty,$ $\|C_{\psi,\phi}f_n\|_{\alpha,m}\leq\frac{l}{2}.$ This is a contradiction, so we conclude that $$\|C_{\psi,\phi}f_n\|_{\alpha,m}\to 0 ~\text{as}~ n\to \infty.$$
		This completes the proof.
	\end{proof}
	
	\section {Essential norm and Hilbert-Schmidt class}\label{S3}
	We now turn our attention to the essential norm of weighted composition operators on $F^2_{\alpha, m}(\mathbb C).$ We begin with the following lemma, which plays a crucial role in the proof of the main result.
	For an entire function $f(z)=\sum\limits_{k=0}^{\infty}a_kz^k,$ define $R_nf(z)=\sum\limits_{k=n}^{\infty}a_kz^k,$ acting on $F^2_{\alpha, m}(\mathbb C)$ and $K_n=I-R_n,$ where $I$ denotes the identity operator on $F^2_{\alpha, m}(\mathbb C).$

	\begin{lemma}\label{Lmmess1}
		Let $f\in F^2_{\alpha, m}(\mathbb C).$ Then there exists positive constant $C$ such that
		\begin{eqnarray*}
			|R_nf(z)|\leq C\|f\|_{\alpha,m}\sum\limits_{j=n}^{\infty}\frac{|w|^j\alpha^{\frac{j-1}{m}}}{\sqrt{\Gamma\left(\frac{2j+2}{m}\right)}},
		\end{eqnarray*}
		for all $z\in\mathbb C.$
	\end{lemma}
	\begin{proof}
		Define $P_{\alpha,m}:L^2(\mu_{\alpha,m})\to F^2_{\alpha, m}(\mathbb C)$ such that 
		$$P_{\alpha,m}f(w)=\frac{m\alpha^{\frac 2m}}{2\pi\Gamma(\frac 2m)}\int_{\mathbb C}f(z)\sum\limits_{n=0}^{\infty}\frac{\alpha^{\frac{2n}{m}}\Gamma(\frac 2m)}{\Gamma\left(\frac{2n+2}{m}\right)}(w\bar z)^ne^{-\alpha |z|^m}dA(z),$$ 
		where $P_{\alpha,m}$ is an orthogonal projection on $F^2_{\alpha, m}(\mathbb C).$ Then 
		\begin{eqnarray*}
			R_nf(w)=P_{\alpha,m}R_nf(w)&=&\frac{m\alpha^{\frac 2m}}{2\pi\Gamma(\frac 2m)}\int_{\mathbb C}R_nf(z)\left(\sum\limits_{n=0}^{\infty}\frac{\alpha^{\frac{2n}{m}}\Gamma(\frac 2m)}{\Gamma\left(\frac{2n+2}{m}\right)}(w\bar z)^n\right)e^{-\alpha |z|^m}dA(z)\\
			&=&\frac{m\alpha^{\frac 2m}}{2\pi\Gamma(\frac 2m)}\int_{\mathbb C}f(z) R_n\left(\sum\limits_{n=0}^{\infty}\frac{\alpha^{\frac{2n}{m}}\Gamma(\frac 2m)}{\Gamma\left(\frac{2n+2}{m}\right)}(w\bar z)^n\right)e^{-\alpha |z|^m}dA(z).
		\end{eqnarray*}
		Now,
		$$\left|R_n\left(\sum\limits_{n=0}^{\infty}\frac{\alpha^{\frac{2n}{m}}\Gamma(\frac 2m)}{\Gamma\left(\frac{2n+2}{m}\right)}(w\bar z)^n\right)\right|\leq \sum\limits_{j=n}^{\infty}\frac{\alpha^{\frac{2j}{m}}\Gamma(\frac 2m)}{\Gamma\left(\frac{2j+2}{m}\right)}|w|^j| z|^j.$$ Thus,
		$$|R_nf(w)|\leq\frac{m\alpha^{\frac 2m}}{2\pi\Gamma(\frac 2m)}\sum\limits_{j=n}^{\infty}\frac{\alpha^{\frac{2j}{m}}\Gamma(\frac 2m)}{\Gamma\left(\frac{2j+2}{m}\right)}|w|^j\int_{\mathbb C}f(z)|z|^j e^{-\alpha |z|^m}dA(z).$$ Using Holder's inequality we get
		$$\int_{\mathbb C}f(z)|z|^j e^{-\alpha |z|^m}dA(z)\leq \left(\int_{\mathbb C}|f(z)|^2e^{-\alpha |z|^m}dA(z)\right)^\frac 12\left(\int_{\mathbb C}|z|^{2j} e^{-\alpha |z|^m}dA(z)\right)^{\frac12}.$$ It follows that
		\begin{eqnarray*}
			|R_nf(w)|
			&\leq&\left(\frac{m\alpha^{\frac 2m}}{2\pi\Gamma(\frac 2m)}\right)^\frac 12\|f\|_{\alpha,m}\sum\limits_{j=n}^{\infty}\frac{\alpha^{\frac{2j}{m}}\Gamma(\frac 2m)}{\Gamma\left(\frac{2j+2}{m}\right)}|w|^j\left(\int_{\mathbb C}|z|^{2j} e^{-\alpha |z|^m}dA(z)\right)^{\frac12}.
		\end{eqnarray*}
		Integration in polar coordinates gives 
		$$\left(\int_{\mathbb C}|z|^{2j} e^{-\alpha |z|^m}dA(z)\right)^{\frac 12}=\left(\frac{2\pi}{m}\right)^{\frac 12}\alpha^{-\frac{j+1}{m}}\Gamma\left(\frac{2j+2}{m}\right)^{\frac 12}.$$
		Therefore, 
		$$ |R_nf(w)|\leq C\|f\|_{\alpha,m}\sum\limits_{j=n}^{\infty}\frac{|w|^j\alpha^{\frac{j-1}{m}}}{\sqrt{\Gamma\left(\frac{2j+2}{m}\right)}},$$ where $C=\alpha^{\frac 1m}\sqrt{\Gamma(\frac 2m)}.$
	\end{proof}

	We are now in a position to establish the main result of this section. We first recall the following characterization of the essential norm. The subsequent theorem provides an asymptotic characterization of the essential norm of bounded weighted composition operators on $F^2_{\alpha, m}(\mathbb C).$
	\begin{lemma}\cite[Prop. 5.1]{S_AM_1987}\label{L1}
		Suppose $T$ is a bounded linear operator on a Hilbert space $\mathcal H.$ Let $\{K_n\}$ be a sequence of a compact self-adjoint operator on $\mathcal H,$ and $R_n=I-K_n.$ Suppose $\|R_n\|=1$ for each $n,$ and $\|R_nx\|_\mathcal H \to 0$ for each $x \in \mathcal H.$ Then $\|T\|_e=\lim\limits_{n \to \infty}\|TR_n\|.$
	\end{lemma}
	\begin{proof}[Proof of Theorem \ref{Th_3}]
		For $w\in\mathbb C,$ define 
		$$F_{\alpha,m}(w)=|\lambda|^{-2}|\psi(\phi^{-1}(w))|^2 e^{\alpha(|w|^m-|\phi^{-1}(w)|^m)}.$$ As $C_{\psi,\phi}$ is bounded, it follows from Theorem \ref{Th_1} that $$\sup_{z\in\mathbb C} \{|\psi(z)|^2 e^{\alpha(|\phi(z)|^m-|z|^m)}: z\in\mathbb C\}<\infty.$$ Thus, $\|F_{\alpha,m}\|_{\infty}=\sup\limits_{w\in\mathbb C}F_{\alpha,m}(w)<\infty.$
		Let $r>0$ and $f\in F^2_{\alpha, m}(\mathbb C)$ with $\|f\|_{\alpha,m}\leq 1.$ 
		Since $R_nf\in F^2_{\alpha, m}(\mathbb C)$ for any positive integer $n,$ then 
		\begin{eqnarray*}
			\|C_{\psi,\phi}R_nf\|_{\alpha,m}^2
			&=&\int_{\mathbb C}|\psi(z)|^2|R_nf(\phi(z))|^2\frac{m\alpha^{\frac 2m}}{2\pi \Gamma(\frac 2m)}e^{-\alpha |z|^m} dA(z)\\
			&=&\frac{m\alpha^{\frac 2m}}{2\pi \Gamma(\frac 2m)}\int_{\mathbb C}|\lambda|^2F_{\alpha,m}(\phi(z))|R_nf(\phi(z))|^2e^{-\alpha |\phi(z)|^m} dA(z).
		\end{eqnarray*}
		By the change of variables $w=\phi(z)=\phi(0)+\lambda z,$ we have
		\begin{eqnarray*}
			\|C_{\psi,\phi}R_nf\|_{\alpha,m}^2&=&\frac{m\alpha^{\frac 2m}}{2\pi \Gamma(\frac 2m)}\int_{\mathbb C}F_{\alpha,m}(w)|R_nf(w)|^2e^{-\alpha |w|^m} dA(w)\\
			&\leq&\|F_{\alpha,m}\|_{\infty}\int_{|w|<r}|R_nf(w)|^2d\mu_{\alpha, m}(w)+\sup_{|w|>r}F_{\alpha,m}(w)\|R_nf\|^2_{\alpha, m}\\
			&\leq&\|F_{\alpha,m}\|_{\infty}\int_{|w|<r}|R_nf(w)|^2d\mu_{\alpha, m}(w)+\sup_{|w|>r}F_{\alpha,m}(w).
		\end{eqnarray*}
		Using Lemma \ref{Lmmess1}, we get
		$$\int_{|w|<r}|R_nf(w)|^2d\mu_{\alpha, m}(w)\leq C \int_{|w|<r}\sum\limits_{j=n}^{\infty} r^j\frac{\alpha^{\frac{j-1}{m}}}{\sqrt{\Gamma\left(\frac{2j+2}{m}\right)}}d\mu_{\alpha, m}(w).$$
		Since $\sum\limits_{j=n}^{\infty} r^j\frac{\alpha^{\frac{j-1}{m}}}{\sqrt{\Gamma\left(\frac{2j+2}{m}\right)}}\to 0$ as $n\to\infty,$ 
		$$\int_{|w|<r}|R_nf(w)|^2d\mu_{\alpha, m}(w)\to 0$$ as $n\to\infty.$ Hence, taking the supremum over $f$ with $\|f\|_{\alpha,m}\leq 1,$ we have 
		$$\|C_{\psi,\phi}R_n\|^2\leq \sup_{|\phi(z)|>r}F_{\alpha,m}(\phi(z))=\sup_{|z|>\frac{r-|\phi(0)|}{|\lambda|}}F_{\alpha,m}(\phi(z)).$$ 
		Letting $r\to\infty,$ we have 
		$$\|C_{\psi,\phi}\|_e \lesssim\lim_{|z|\to\infty}\sup |\psi(z)|^2 e^{\alpha(|\phi(z)|^m-|z|^m)}.$$

		Next we prove the lower estimate for the essential norm. Since $C_{\psi,\phi}$ is bounded, $\phi(z)=\phi(0)+\lambda z,~ |\lambda|\leq1.$ Thus, $|\phi(z)|\to\infty$ as $|z|\to\infty.$
		As $\{\hat k_{z}^{\alpha,m}\}\to 0$ weakly as $|z|\to\infty,$ we see that $\|K \hat k_{\phi(z)}^{\alpha,m}\|_{\alpha,m}^2\to 0$ as $|z|\to\infty.$  
		Therefore, 
		\begin{eqnarray*}
			\|C_{\psi,\phi}-K\|
			&\geq& \lim_{|z|\to\infty}\sup\|(C_{\psi,\phi}-K)\hat k_{\phi(z)}^{\alpha,m}\|_{\alpha,m}\\
			&\geq&\lim_{|z|\to\infty}\sup (\|C_{\psi,\phi}\hat k_{\phi(z)}^{\alpha,m}\|_{\alpha,m}-\|K\hat k_{\phi(z)}^{\alpha,m}\|_{\alpha,m})\\
			&=&\lim_{|z|\to\infty}\sup \|C_{\psi,\phi}\hat k_{\phi(z)}^{\alpha,m}\|_{\alpha,m}.
		\end{eqnarray*}
		Thus, 
		\begin{eqnarray*}
			\|C_{\psi,\phi}\|_{e}^2
			&\geq&\lim_{|z|\to\infty}\sup \|C_{\psi,\phi}\hat k_{\phi(z)}^{\alpha,m}\|_{\alpha,m}^2\\
			&=& \lim_{|z|\to\infty}\sup \|\psi(z)\hat k_z^{\alpha,m}(\phi(z))\|_{\alpha,m}^2\\
			&\geq&\lim_{|z|\to\infty}\sup |\psi(z)|^2\frac{\|k^{\alpha,m}_{\phi(z)}\|^2}{\|k_z^{\alpha,m}\|^2}\\
			&\asymp &\lim_{|z|\to\infty}\sup |\psi(z)|^2\left|\frac{\phi(z)}{z}\right|^{m-2}e^{\alpha(|\phi(z)|^m-|z|^m)}.
		\end{eqnarray*}
		Since $C_{\psi,\phi}$ is bounded, it follows from Theorem \ref{Th_1} that for $|\lambda|\in(0,1],$ 
		$$\lim_{|z|\to\infty}\sup |\psi(z)|^2\left|\frac{\phi(z)}{z}\right|^{m-2}e^{\alpha(|\phi(z)|^m-|z|^m)}\asymp \lim_{|z|\to\infty}\sup |\psi(z)|^2e^{\alpha(|\phi(z)|^m-|z|^m)}.$$ Therefore,
		$$ \|C_{\psi,\phi}\|_{e}^2\gtrsim \lim_{|z|\to\infty}\sup |\psi(z)|^2e^{\alpha(|\phi(z)|^m-|z|^m)}.$$
	\end{proof}

	We conclude this section by proving the characterization of Hilbert-Schmidt weighted composition operators.
	\begin{proof}[Proof of Theorem \ref{Th_4}]
		Let us consider an orthonormal basis $\{e_n\}_{n=0}^\infty$ on $F^2_{\alpha, m}(\mathbb C),$ defined by $e_n(z)=\sqrt{\frac{\alpha^{\frac{2n}{m}}\Gamma(\frac 2m)}{\Gamma{(\frac{2n+2}{m})}}}z^n.$ Then $C_{\psi,\phi}$ is Hilbert-Schmidt on $F^2_{\alpha, m}(\mathbb C)$ if and only if 
		$$\sum\limits_{n=0}^\infty\|C_{\psi,\phi} e_n\|^2_{{\alpha, m}}<\infty.$$ 
		Now,
		\begin{align*}
			\sum\limits_{n=0}^\infty\|C_{\psi,\phi} e_n\|^2_{\alpha, m}
			&=\sum\limits_{n=0}^\infty\frac{m \alpha ^{\frac2m}}{2\pi\Gamma(\frac 2m)}\int_{\mathbb C} |\psi(z)|^2|e_n(\phi(z)|^2 e^{-\alpha |z|^m}dA(z)\\
			&=\sum\limits_{n=0}^\infty\frac{m \alpha ^{\frac2m}}{2\pi\Gamma(\frac 2m)}\int_{\mathbb C} |\psi(z)|^2 \frac{\alpha^{\frac{2n}{m}\Gamma(\frac 2m)}}{\Gamma(\frac{2n+2}{m})} |\phi(z)|^{2n}  e^{-\alpha |z|^m}dA(z)\\
			&=\frac{m \alpha ^{\frac2m}}{2\pi\Gamma(\frac 2m)}\int_{\mathbb C} |\psi(z)|^2 \sum\limits_{n=0}^\infty\frac{\alpha^{\frac{2n}{m}\Gamma(\frac 2m)}}{\Gamma(\frac{2n+2}{m})} |\phi(z)|^{2n}  e^{-\alpha |z|^m}dA(z)\\
			&\asymp  \frac{m^2 \alpha}{4\pi\Gamma(\frac 2m)}\int_{\mathbb C} |\psi(z)|^2 |\phi(z)|^{m-2} e^{\alpha (|\phi(z)|^m-|z|^m)}dA(z).
		\end{align*}
		Therefore, $C_{\psi,\phi}$ is Hilbert-Schmidt on $F^2_{\alpha, m}(\mathbb C)$ if and only if $$\int_{\mathbb C} |\psi(z)|^2 |\phi(z)|^{m-2} e^{\alpha (|\phi(z)|^m-|z|^m)}dA(z)<\infty.$$
	\end{proof}

	\section {Additional observations}\label{S4}
	The purpose of this section is to illustrate the applicability of the main results obtained in the previous sections. We present several examples and special cases, derive explicit criteria for weighted composition operators induced by exponential weights, and obtain corresponding results for composition operators.
	\begin{proposition}\label{pro_1}
		Let $\phi(z)=\phi(0)+\lambda z$ and $\psi(z)=ce^{dz},$ where $c,d \in\mathbb C$ and $c\neq0.$
		\begin{itemize}
			\item[(i)] For $m\in(0,1),$ $C_{\psi,\phi}$ is bounded on $F^2_{\alpha, m}(\mathbb C)$if and only if $|\lambda|\leq1$ and $d=0.$
			\item [(ii)] For $m>1,$~$C_{\psi,\phi}$ is bounded on $F^2_{\alpha, m}(\mathbb C)$ if $|\lambda|<1.$
		\end{itemize}
	\end{proposition}
	
	\begin{proof}
		(i) Let $0<m<1$ and $C_{\psi,\phi}$ is bounded on $F^2_{\alpha, m}(\mathbb C).$ By Theorem \ref{Th_1}, we have $|\lambda|\leq1$ and $\psi\in F^2_{\alpha, m}(\mathbb C).$ Since $\psi(z)=ce^{dz},$ it follows that $\psi\in F^2_{\alpha, m}(\mathbb C)$ if and only if $d=0.$ Conversely, assume that $|\lambda|\leq1$ and $d=0.$ Then $\psi(z)=c.$ Now, we only show that $$\sup\limits_{z\in\mathbb C}e^{\alpha(|\phi(z)|^m-|z|^m)}<\infty.$$ Since $\phi(z)=\phi(0)+\lambda z,$ we have $\lim\limits_{|z|\to\infty}\frac{|\phi(z)|-|\lambda||z|}{|z|}=0$ and hence
		$$|\phi(z)|^m-|z|^m=|z|^m(|\lambda|^m-1)+o(|z|^m).$$ By the definition of $o(|z|^m),$ for every $\epsilon>0$ there exists $R_0>0$ such that $$|o(|z|^m)|\leq\epsilon|z|^m,~~|z|\geq R_0.$$ Choose $\epsilon=\frac{(1-|\lambda|^m)}{2}.$ Thus, for $|\lambda|<1,$ 
		$$e^{\alpha(|\phi(z)|^m-|z|^m)}\leq e^{-\frac{\alpha}{2}(1-|\lambda|^m)|z|^m}\leq1.$$ 
		It follows that 
		$$\sup\limits_{|z|>R_0}e^{\alpha(|\phi(z)|^m-|z|^m)}<\infty.$$ On the other hand, the function $e^{\alpha(|\phi(z)|^m-|z|^m)}$ is bounded on $\{z:|z|\leq R_0\}$. Thus, $$\sup\limits_{z\in\mathbb C}e^{\alpha(|\phi(z)|^m-|z|^m)}<\infty.$$ For $|\lambda|=1,$ 
		we have  $$|\phi(z)|^m-|z|^m=m|z|^{m-1}\Re\left(\frac{\phi(0)}{\lambda e^{i\theta}}\right)+O(|z|^{m-2}).$$ As $0<m<1$ and $\Re\left(\frac{\phi(0)}{\lambda e^{i\theta}}\right)\leq|\phi(0)|,$ we have $$|z|^{m-1}\Re\left(\frac{\phi(0)}{\lambda e^{i\theta}}\right)\leq |z|^{m-1}|\phi(0)|\to0~~~\text{as}~|z|\to\infty.$$
		Thus, $$\lim\limits_{|z|\to\infty}(|\phi(z)|^m-|z|^m)=0,$$
		and so $$\lim\limits_{|z|\to\infty}e^{\alpha(|\phi(z)|^m-|z|^m)}=1.$$ Therefore, the boundedness of $C_{\psi,\phi}$ follows. \\
		(ii) First consider $|\lambda|<1.$ By Theorem \ref{Th_1}, it is enough to verify that
		$$\sup_{z\in\mathbb C}e^{2\Re(dz)+\alpha(|\phi(z)|^m-|z|^m)}<\infty.$$  We have
		$$|\phi(z)|^m-|z|^m=|z|^m(|\lambda|^m-1)+o(|z|^m).$$ By the similar argument as in (i), we get
		$$e^{2\Re(dz)+\alpha(|\phi(z)|^m-|z|^m)}\leq e^{2|d||z|-\frac{\alpha (1-|\lambda|^m)}{2}|z|^m}\leq e^{\frac{2(m-1)|d|}{m}\left(\frac{4|d|}{\alpha m (1-|\lambda|^m)}\right)^{\frac{1}{m-1}}},~~|z|\geq R_0.$$ It follows that 
		$$\sup\limits_{|z|>R_0}e^{2\Re(dz)+\alpha(|\phi(z)|^m-|z|^m)}<\infty.$$ Also, the function $e^{2\Re(dz)+\alpha(|\phi(z)|^m-|z|^m)}$ is bounded on $\{z:|z|\leq R_0\}$. Thus, $$\sup\limits_{z\in\mathbb C}e^{2\Re(dz)+\alpha(|\phi(z)|^m-|z|^m)}<\infty$$ and the boundedness of $C_{\psi,\phi}$ follows.
	\end{proof}

	\begin{remark}
		(i)~~The converse of Proposition~\ref{pro_1}(ii) does not hold. Indeed, consider the case
		\[
		\phi(0)=0,\qquad d=0,\qquad |\lambda|=1.
		\]
		Then \(\phi(z)=\lambda z\) and \(\psi\) is a constant function. Consequently,
		\[
		|\psi(z)|^2e^{\alpha\left(|\phi(z)|^m-|z|^m\right)}
		=c^2,
		\]
		which is bounded on \(\mathbb{C}\). Hence, \(C_{\psi,\phi}\) is bounded. On the other hand, the condition $|\lambda|=1$ does not guarantee boundedness. For example, consider the case
		$$ \phi(0)=0,~c=1,d\neq0,~ \text{and}~|\lambda|=1.$$ Then $\phi(z)=\lambda z$ and $\psi(z)=e^{dz}.$ Consequently,
		
		\[|\psi(z)|^2e^{\alpha\left(|\phi(z)|^m-|z|^m\right)}
		=e^{2\Re(dz)}.
		\]
		Since $d\neq0,$ choose $z=t\frac{\bar d}{|d|},~ t>0.$ Then $\Re(dz)=t|d|.$ Thus, we obtain
		$$|\psi(z)|^2e^{\alpha\left(|\phi(z)|^m-|z|^m\right)}=e^{2t|d|}\to\infty ~~\text{as}~~ t \to \infty.$$ Therefore, \(C_{\psi,\phi}\) is not bounded. \\
		One may conjecture from the above examples that, for \(m>1\), if \(C_{\psi,\phi}\) is bounded and \(|\lambda|=1\), then necessarily \(\phi(0)=0\) and \(d=0\). However, this is not true in general. To see this, consider $$m=2,~ \phi(0)\neq0, c=1, d=-\alpha\overline{\phi(0)} ~\text{and} ~|\lambda|=1.$$ Then
		\[
		|\psi(z)|^2=e^{-2\alpha\Re(\overline{\phi(0)}z)},
		\]
		and
		\[
		|\phi(z)|^2-|z|^2
		=2\Re(\overline{\phi(0)}z)+|\phi(0)|^2.
		\]
		Hence,
		\[
		\begin{aligned}
			|\psi(z)|^2e^{\alpha\left(|\phi(z)|^2-|z|^2\right)}
			=e^{-2\alpha\Re(\overline{\phi(0)}z)}
			e^{\alpha\left(2\Re(\overline{\phi(0)}z)+|\phi(0)|^2\right)}
			=e^{\alpha|\phi(0)|^2}.
		\end{aligned}
		\]
		Therefore,
		\[
		\sup_{z\in\mathbb{C}}
		|\psi(z)|^2e^{\alpha\left(|\phi(z)|^2-|z|^2\right)}<\infty.
		\]
		Thus, \(C_{\psi,\phi}\) is bounded on \(F^2_{\alpha,2}(\mathbb{C})\), although \(\phi(0)\neq0\) and \(d\neq0\).\\
		(ii)~~The preceding examples show that, unlike the case $m>1,$ the boundedness of the weighted composition operator on the weighted Fock space corresponding to $m=1$ cannot be characterized solely in terms of the symbol $\phi(z)=\phi(0)+\lambda z.$ In particular, the condition $|\lambda|<1$ does not guarantee boundedness. For example, if 
		$$\phi(z)=\frac{z}{2},~~\psi(z)=e^{\frac{2z}{5}}~~\text{with}~\alpha=1,$$ then 
		\[
		|\psi(z)|^2e^{\alpha\left(|\phi(z)|-|z|\right)}=e^{\frac45 \Re(z)-\frac 12 z},
		\]
		which tends to $\infty$ along the positive real axis. Hence $C_{\psi,\phi}$ is not bounded despite the fact that $|\lambda|<1.$ 
		Likewise, the example 
		$$\phi(z)=z+1,~~\psi(z)=e^{-\frac z4}~~\text{with}~\alpha=1,$$ shows that boundedness may also fail when $|\lambda|=1.$ Therefore, in the case $m=1$, no characterization of boundedness can be expressed solely in terms of the parameter $|\lambda|.$
	\end{remark}

	\begin{proposition}\label{pro_2}
		Let $\phi(z)=\phi(0)+\lambda z$ and $\psi(z)=ce^{dz},$ where $c,d \in\mathbb C$ and $c\neq0.$
		\begin{itemize}
			\item [(i)]~For $0<m<1,$ $C_{\psi,\phi}$ is compact on $F^2_{\alpha, m}(\mathbb C)$ if and only if $|\lambda|<1$ and $d=0.$
			\item[(ii)]~For $m>1,$ $C_{\psi,\phi}$ is compact on $F^2_{\alpha, m}(\mathbb C)$ if and only if $|\lambda|<1.$
		\end{itemize}
		
	\end{proposition}
	\begin{proof}
		(i)~~Let $0<m<1.$ First consider that $|\lambda|<1$ and $d=0.$ We only show that $$\lim\limits_{|z|\to\infty}e^{\alpha(|\phi(z)|^m-|z|^m)}=0.$$  
		Now, $$e^{\alpha(|\phi(z)|^m-|z|^m)}\leq e^{-\frac{\alpha}{2}(1-|\lambda|^m)|z|^m}.$$ Since $|\lambda|<1,$
		$$\lim\limits_{|z|\to\infty}e^{-\frac{\alpha}{2}(1-|\lambda|^m)|z|^m}=0$$ and hence the result follows. Conversely, let $C_{\psi,\phi}$ is compact on $F^2_{\alpha, m}(\mathbb C).$ Then it follows from Theorem \ref{Th_2} that $|\lambda|<1$ and $d=0.$\\
		(ii)~~Let $m>1.$ Using the same argument as in Proposition \ref{pro_1} (ii), we get 
		$$e^{2\Re(dz)+\alpha(|\phi(z)|^m-|z|^m)}\leq e^{2|d||z|-\frac{\alpha (1-|\lambda|^m)}{2}|z|^m}.$$ Thus, 
		$$\lim\limits_{|z|\to\infty}e^{2\Re(dz)+\alpha(|\phi(z)|^m-|z|^m)}=0$$
		and so $C_{\psi,\phi}$ is compact on $F^2_{\alpha, m}(\mathbb C).$  The converse is an immediate consequence of Theorem~\ref{Th_2}.
	\end{proof}
	
	\begin{remark}
		For the case $m=1,$ consider the weighted composition operator
		$C_{\psi,\phi}$ on $F^2_{1,1}(\mathbb C)$, where
		\[
		\phi(z)=\frac{z}{2}, \qquad \psi(z)=e^{z/4},
		\] with $\alpha=1.$ Here, $|\lambda|<1.$
		By the boundedness characterization, $C_{\psi,\phi}$ is bounded on
		$F^2_{1,1}(\mathbb C).$ However, $C_{\psi,\phi}$ is not compact since
		\[
		\lim_{t\to\infty} |\psi(t)|^2
		e^{|\phi(t)|-|t|}
		=
		\lim_{t\to\infty}
		e^{t/2+t/2-t}
		=1\neq0.
		\]
		Therefore, this example shows that for $m=1,$ the condition
		$|\lambda|<1$ alone is not sufficient to guarantee the compactness of
		$C_{\psi,\phi}.$ 
	\end{remark}

	For the composition operator, our results can be reduced to the following series of equivalent assertions.
	
	\begin{cor}
		Let $\phi(z)=\phi(0)+\lambda z,$ where $\phi(0), \lambda\in\mathbb C.$ 
		\begin{itemize}
			\item [(i)]For $0<m\leq1,$~$C_{\phi}$ is bounded on $F^2_{\alpha, m}(\mathbb C)$ if and only if $|\lambda|\leq1.$ 
			\item[(ii)]For $m>1,$ $C_{\phi}$ is bounded on $F^2_{\alpha, m}(\mathbb C)$ if and only if either $|\lambda|<1$ or $|\lambda|=1$ with $\phi(0)=0$.
		\end{itemize}
	\end{cor}
	\begin{proof}
		(i)~Let $0<m\leq1.$ Suppose that $C_{\phi}$ is bounded on $F^2_{\alpha, m}(\mathbb C).$ Then Theorem \ref{Th_1} implies that $|\lambda|\leq1.$ Conversely, let $|\lambda|\leq1.$ For $0<m<1,$ it follows from Proposition~\ref{pro_1} (i) that $C_\phi$ is bounded on $F^2_{\alpha, m}(\mathbb C)$. Now, for $m=1,$ we have $$|\phi(z)|-|z|\leq|\phi(0)|.$$ Thus,
		$$e^{\alpha(|\phi(z)|-|z|)}\leq e^{\alpha |\phi(0)|},$$ and so $$\sup\limits_{z\in\mathbb C}e^{\alpha(|\phi(z)|-|z|)}<\infty.$$ Therefore, $C_{\phi}$ is bounded on $F^2_{\alpha, m}(\mathbb C).$\\
		(ii)~Let $m>1.$ Suppose $C_{\phi}$ is bounded on $F^2_{\alpha, m}(\mathbb C).$ By Theorem \ref{Th_1}, we get $|\lambda|\leq1.$ For the case $|\lambda|=1,$ assume that $\phi(0)\neq 0.$ Consider $z=t\lambda^{-1}\phi(0),~t>0.$ Then $\phi(z)=(t+1)\phi(0)$ and hence 
		$$|\phi(z)|^m-|z|^m=((t+1)^m-t^m)|\phi(0)|^m.$$ As
		$$\lim\limits_{t\to\infty}\frac{(t+1)^m-t^m}{t^{m-1}}=m,$$ it follows that $(t+1)^m-t^m\asymp t^{m-1}$. Since $m>1,$ we have $t^{m-1}\to\infty$ and therefore $$(t+1)^m-t^m\to\infty~~\quad\text{as}~ t\to\infty.$$ Consequently, $$|\phi(z)|^m-|z|^m\to\infty~~~~\quad\text{as}~~|z|\to\infty,$$ which contradicts the boundedness criterion in Theorem~\ref{Th_1}. Hence, $\phi(0)=0.$ \\
		Conversely, if $|\lambda|<1$ then it follows from Proposition \ref{pro_1} (ii) that $C_{\phi}$ is bounded on $F^2_{\alpha, m}(\mathbb C).$ Again, if $|\lambda|=1$ with $\phi(0)=0$ then $|\phi(z)|=|z|.$ Thus, $$\sup\limits_{z\in\mathbb C}e^{\alpha(|\phi(z)|^m-|z|^m)}=1$$ and therefore, $C_\phi$ is bounded. 
	\end{proof}

	\begin{cor}
		Let $\phi(z)=\phi(0)+\lambda z,$ where $\phi(0), \lambda\in\mathbb C.$ Then the following are equivalent.
		\begin{itemize}
			\item [(i)] $C_{\phi}$ is compact on $F^2_{\alpha, m}(\mathbb C)$.
			\item[(ii)] $C_{\phi}$ is Hilbert-Schmidt on $F^2_{\alpha, m}(\mathbb C).$  
			\item [(iii)] $|\lambda|<1.$
		\end{itemize}
	\end{cor}
	
	\begin{proof}
		The equivalence $(i)\Longleftrightarrow(iii)$ follows immediately from Proposition \ref{pro_2} for $m\neq1$. For $m=1,$ suppose that $C_{\phi}$ is compact on $F^2_{\alpha, m}(\mathbb C)$. Then $|\lambda|\leq1.$ We claim that $|\lambda|<1.$ Indeed, if $|\lambda|=1,$ then $|\phi(z)|-|z|\geq-|\phi(0)|.$ Thus, 
		$$e^{\alpha(|\phi(z)|-|z|)}\geq e^{-\alpha|\phi(0)|}>0,$$ which contradicts the compactness criterion. Conversely, suppose that $|\lambda|<1.$ Then 
		$$|\phi(z)|-|z|\leq\phi(0)-(1-|\lambda|)|z|.$$ Hence, 
		$$e^{\alpha(|\phi(z)|-|z|)}\leq e^{\alpha(\phi(0)-(1-|\lambda|)|z|)}\to0\quad\quad \text{as}~|z|\to\infty.$$
		Therefore, $C_{\phi}$ is compact on $F^2_{\alpha, m}(\mathbb C)$.
		Moreover, since every Hilbert--Schmidt operator is compact, we have
		$(ii)\Longrightarrow(i).$
		
		It remains to prove that $(i)\Longrightarrow(ii)$.
		Assume that $C_{\phi}$ is compact on $F^2_{\alpha, m}(\mathbb C)$. By $(i)\Longleftrightarrow(iii)$, we have $|\lambda|<1$. If $\lambda=0$ then $\phi(z)=\phi(0)\neq 0$ and 
		$$\int_{\mathbb C} |\phi(z)|^{m-2} e^{\alpha ({|\phi(z)|^m-|z|^m})}dA(z)=|\phi(0)|^{m-2}e^{\alpha|\phi(0)|^m}
		\int_{\mathbb C}e^{-\alpha|z|^m}\,dA(z)<\infty.$$ Hence, $C_{\phi}$ is Hilbert-Schmidt. Now assume that $\lambda\neq 0.$ Since $\phi(z)=\phi(0)+\lambda z,$ $|\phi(z)|\asymp |\lambda| |z|.$ Then there exist positive constants $C$ and $R_0$ such that $$|\phi(z)|^{m-2}\leq C|\lambda|^{m-2} |z|^{m-2},~~\quad|z|>R_0.$$ We have
		$$|\phi(z)|^m-|z|^m=|z|^m(|\lambda|^m-1)+o(|z|^m).$$ By the definition of $o(|z|^m),$ for every $\epsilon>0$ there exists $R_0>0$ such that $$|o(|z|^m)|\leq\epsilon|z|^m,~~|z|\geq R_1.$$ Choose $\epsilon=\frac{(1-|\lambda|^m)}{2}.$ Thus,
		$$e^{\alpha(|\phi(z)|^m-|z|^m)}\leq e^{-\frac{\alpha}{2}(1-|\lambda|^m)|z|^m},\quad |z|>R_1.$$
		Hence, 
		$$|\phi(z)|^{m-2} e^{\alpha ({|\phi(z)|^m-|z|^m})}\leq (C|\lambda|)^{m-2}|z|^{m-2}e^{-\frac{\alpha}{2}(1-|\lambda|^m)|z|^m},\quad |z|>R=\max\{R_0, R_1\}.$$
		Take $c_0=\frac{\alpha}{2}(1-|\lambda|^m)|>0.$ Now,
		$$\int_{|z|>R}|z|^{m-2}e^{-c_0|z|^m}dA(z)=\frac{2\pi}{c_0m}e^{-c_0R^m}.$$ It follows that 
		$$\int_{|z|>R} |\phi(z)|^{m-2} e^{\alpha ({|\phi(z)|^m-|z|^m})}dA(z)<\infty.$$ On the other hand, the integrand is locally integrable on the disk $\{z:|z|\leq R\}.$ Therefore, $$\int_{\mathbb C} |\phi(z)|^{m-2} e^{\alpha ({|\phi(z)|^m-|z|^m})}dA(z)<\infty.$$
	\end{proof}

	\section{Examples}\label{S5}
	
	In this section, we provide examples showing that weighted composition operators may have boundedness, compactness, or the Hilbert–Schmidt property on the classical Fock space but not on weighted Fock spaces, and conversely, may have these properties on weighted spaces but not on the classical one.

	\begin{table}[h]
		\centering
		
		\textbf{Table 1}
		
		\vspace{0.3cm}
		\renewcommand{\arraystretch}{1.5}
		\begin{tabular}{|c|c|c|c|c|}
			\hline
			\textbf{Functions} & \textbf{m} & $\mathbf{Bounded}$ & $\mathbf{Compact}$ & $\mathbf{Hilbert-Schmidt}$ \\
			\Xhline{1.2pt}
			
			\multirow{3}{*}{$
				\begin{aligned}
					\phi(z)&=\frac{z}{2},\\
					\psi(z)&=e^{z^3}
				\end{aligned}
				$}
			& $\frac{1}{2}$ & $\times$ & $\times$ & $\times$ \\ \cline{2-5}
			& 2 & $\times$ & $\times$ & $\times$ \\ \cline{2-5}
			& 4 & $\checkmark$ & $\checkmark$ & $\checkmark$ \\
			\Xhline{1.2pt}
			
			\multirow{3}{*}{$
				\begin{aligned}
					\phi(z)&=\frac{z}{2},\\
					\psi(z)&=e^{z^2}
				\end{aligned}
				$}
			& $\frac{1}{2}$ & $\times$ & $\times$ & $\times$ \\ \cline{2-5}
			& 2 & $\checkmark$ & $\checkmark$ & $\checkmark$ \\ \cline{2-5}
			& 4 & $\checkmark$ & $\checkmark$ & $\checkmark$ \\
			\Xhline{1.2pt}
			
			\multirow{3}{*}{$
				\begin{aligned}
					\phi(z)&=z+1,\\
					\psi(z)&=e^{-z}
				\end{aligned}
				$}
			& $\frac{1}{2}$ & $\times$ & $\times$ & $\times$ \\ \cline{2-5}
			& 2 & $\checkmark$ & $\times$ & $\times$ \\ \cline{2-5}
			& 4 & $\times$ & $\times$ & $\times$ \\
			\Xhline{1.2pt}
			
			\multirow{3}{*}{$
				\begin{aligned}
					\phi(z)&=\frac{z}{2},\\
					\psi(z)&=c\neq0
				\end{aligned}
				$}
			& $\frac{1}{2}$ & $\checkmark$ & $\checkmark$ & $\checkmark$ \\ \cline{2-5}
			& 2 & $\checkmark$ & $\checkmark$ & $\checkmark$ \\ \cline{2-5}
			& 4 & $\checkmark$ & $\checkmark$ & $\checkmark$ \\
			\Xhline{1.2pt}
			
		\end{tabular}
		
	\end{table}

	Our examples demonstrate two things: first, boundedness, compactness, and the Hilbert-Schmidt property are not weight-invariant; second, the classical Fock space is exceptional, not the rule. Indeed, operators may possess these properties on the classical space yet lose them on weighted counterparts, and conversely.

	\section{Concluding remarks}\label{S6}
	
	The results obtained in this paper reveal an interesting connection between compactness and the Hilbert--Schmidt property of weighted composition operators on the weighted Fock space $F^2_{\alpha,m}(\mathbb C)$. In particular, for the special class of weighted composition operators induced by
	\[
	\psi(z)=ce^{dz}, \qquad \phi(z)=\lambda z,\qquad c,d\in\mathbb C,\; |\lambda|<1,
	\]
	our compactness and Hilbert-Schmidt characterizations show that every compact weighted composition operator is automatically Hilbert-Schmidt for every $m>0$. 
	This observation naturally raises the following question.
	
	\medskip
	
	\begin{question}\label{q1}
		Let $m>0$. Does every compact weighted composition operator on $F^2_{\alpha,m}(\mathbb C)$ necessarily belong to the Hilbert--Schmidt class?
	\end{question}

	\medskip

	In $F^2_{\alpha}(\mathbb C^n)$, it is known that the operator norm and the essential norm of every non-compact composition operator are equal (see \cite{JPZ_JMAA_17}). Motivated by this result and Theorem~\ref{Th_3}, which provides a characterization of the essential norm of weighted composition operators on $F^2_{\alpha,m}(\mathbb C)$, it is natural to ask the following.
	
	\medskip
	
	\begin{question}\label{q2}
		Let $C_{\psi,\phi}$ be a bounded, non-compact weighted composition operator on $F^2_{\alpha,m}(\mathbb C)$. Does the relation
		$
		\|C_{\psi,\phi}\|
		\asymp
		\|C_{\psi,\phi}\|_e
		$ hold?
	\end{question}

	An affirmative answer to Question \ref{q1} would establish a close relationship between compactness and the Hilbert-Schmidt property for weighted composition operators on  $F^2_{\alpha,m}(\mathbb C)$. Similarly, an affirmative answer to Question \ref{q2} would extend the asymptotic equivalence between the operator norm and the essential norm, known for composition operators on the classical Fock space, to the broader setting of weighted composition operators on $F^2_{\alpha,m}(\mathbb C)$. On the other hand, negative answers to either of these questions would reveal new phenomena and highlight fundamental differences between generalized weighted Fock spaces and the classical Fock space. We hope that these questions will motivate further research on the compactness, Schatten class properties, and essential norm theory of weighted composition operators on generalized weighted Fock spaces.

	\section*{Declarations}	
	\textit{Acknowledgements.} This work originated during a research visit of Dr. Anirban Sen to the Universidad Complutense de Madrid. He sincerely thanks Prof. E. A. Gallardo-Guti\'errez and the department for their generous hospitality and inspiring academic atmosphere, which were invaluable to the progress of this research.
	Dr. Anirban Sen is supported by the Czech Science Foundation (GA CR) grant no. 25-18042S. Miss Somdatta Barik would like to thank UGC, Govt. of India, for the financial support in the form of Senior Research Fellowship under the mentorship of Prof. Kallol Paul.
		\\
		\textit{Author Contributions:} All authors contributed equally to this manuscript and approved the final version.\\
		\textit{Data Availability :} No data was used for the research described in this article.\\
		\textit{Conflict of interest:} The authors declare that they have no conflict of interest.

\end{document}